\documentclass[12pt,reqno]{amsart}
\usepackage[T1]{fontenc}
\usepackage[utf8]{inputenc}
\usepackage{amssymb, amsthm, amsxtra, epsfig, amsmath, parskip}
\usepackage{hyperref}

\newcommand{\ind}{{1\hspace{-1mm}{\rm I}}}
\newcommand{\N}{\mathbb{N}}

\newcommand{\R}{\mathbb{R}}

\newcommand{\E}{\mathbb{E}}

\newtheorem{theorem}{Theorem}[section]

\theoremstyle{definition}

\theoremstyle{remark}

\begin{document}

\sloppy
\title[Infinite divisibility of random fields]{Infinite divisibility of random fields admitting an integral representation with an infinitely divisible integrator}

\author{Wolfgang Karcher}
\address{Wolfgang Karcher, Ulm University, Institute of Stochastics, Helmholtzstr. 18, 89081 Ulm, Germany}
\email{wolfgang.karcher\@@{}uni-ulm.de}
\author{Hans-Peter Scheffler}
\address{Hans-Peter Scheffler, University of Siegen, Fachbereich 6, Mathematik, Emmy-Noether-Campus, Walter-Flex-Str. 3, 57068 Siegen, Germany}
\email{scheffler\@@{}mathematik.uni-siegen.de}
\author{Evgeny Spodarev}
\address{Evgeny Spodarev, Ulm University, Institute of Stochastics, Helmholtzstr. 18, 89081 Ulm, Germany}
\email{evgeny.spodarev\@@{}uni-ulm.de}

\date{7 October 2009}

\begin{abstract}
Let $\Lambda$ be an infinitely divisible random measure. We consider random fields of the form
$$ X(t) = \int_{\R^d} f_t(x) \Lambda(dx), \quad t \in \R^q, \quad d,q \geq 1,$$
where $f_t:\R^d \to \R$ is $\Lambda$-integrable for all $t \in \R^d$. We show that $X$ is an infinitely divisible random field, that is the law of the random vector $(X(t_1),...,X(t_n))$ is an infinitely divisible probability measure on $\R^n$ for all $t_1,...,t_n \in \R^q$.
\end{abstract}


\maketitle

\baselineskip=18pt

\section{Preliminaries}

We start with the definition of an infinitely divisible random measure. Following \cite{HPVJ08}, let $R$ be a Borel subset of $\R^d$, $\mathcal{B}(R)$ be the Borel sets contained in $R$, and $\mathcal{S}$ be the $\delta$-ring (a ring closed under countable intersections) of bounded subsets of $R$. Let $\Lambda = \{\Lambda(A), A \in \mathcal{S}\}$ be a stochastic process with the following three properties.
\begin{itemize}
 \item $\Lambda$ is \textit{independently scattered}: If $\{A_n\}_{n \in \N} \subset \mathcal{S}$ is a sequence of disjoint sets, then the random variables $\Lambda(A_n)$, $n \in \N$, are independent.
 \item $\Lambda$ is \textit{$\sigma$-additive}: If $\{A_n\}_{n \in \N} \subset \mathcal{S}$ is a sequence of disjoint sets and $\bigcup\limits_n A_n \in \mathcal{S}$, then
	$$\Lambda(\bigcup_n A_n) = \sum_n \Lambda(A_n) \quad a.s.$$
 \item $\Lambda(A)$ is an infinitely divisible random variable for each $A \in \mathcal{S}$, i.~e. $\Lambda(A)$ has the law of the sum of $n$ independent identically distributed random variables for any $n \geq 1$.
\end{itemize}
Then $\Lambda$ is called \textit{infinitely divisible random measure}.

We now consider the \textit{cumulant function} $C_{\Lambda(A)}(t) = \ln(\E e^{it\Lambda(A)})$ of $\Lambda(A)$ for a set $A$ in $\mathcal{S}$ which is given by the \textit{L\'evy-Khintchine representation}
$$C_{\Lambda(A)}(t) = ita(A) - \frac{1}{2}t^2b(A) + \int_\R \left(e^{itr}-1-it\tau(r)\right)F(dr,A),$$
where $a$ is a $\sigma$-additive set function on $\mathcal{S}$, $b$ is a measure on $\mathcal{B}(R)$, and $F(dr,A)$ is a measure on $\mathcal{B}(R)$ for fixed $dr$ and a \textit{L\'evy measure} on $\mathcal{B}(\R)$ for each fixed $A \in \mathcal{B}(R)$, that is $F(\{0\},A) = 0$ and $\int_\R \min\{1,r^2\} F(dr,A) < \infty$, and $\tau(r) = r \ind_{[-1,1]}(r)$. $F$ is a measure and referred to as the \textit{generalized L\'evy measure} and $(a,b,F)$ is called \textit{characteristic triplet}.

Let $|a| = a^+ + a^-$.The measure $\lambda$ with
\begin{equation*}
 \lambda(A) := |a|(A) + b(A) + \int_\R \min\{1,r^2\} F(dr,A), \quad A \in \mathcal{S},
\end{equation*}
is called \textit{control measure} of the infinitely divisible random measure $\Lambda$.

Let $f_t:\R^d \to \R$, $d \geq 1$, be \textit{$\Lambda$-integrable} for all $t \in \R^q$, $q \geq 1$, that is there exists a sequence of simple functions $\{\tilde{f}_t^{(n)}\}_{n \in \N}$, $\tilde{f}_t^{(n)}:\R^d \to \R$, $t \in \R^q$, such that
\begin{enumerate}
 \item $\tilde{f}_t^{(n)} \to f_t \quad \lambda-\text{a.e.},$
 \item for every Borel set $B \in \R^d$, the sequence $\{\int\limits_{B} \tilde{f}_t^{(n)}(x) \Lambda(dx)\}_{n \in \N}$ converges in probability.
\end{enumerate}
For all $t \in \R^q$, we define
$$ \int\limits_{\R^d} f_t(x) \Lambda(dx) := \underset{n \to \infty}{\text{plim}} \int\limits_{\R^d} \tilde{f}_t^{(n)}(x) \Lambda(dx),$$
where $\underset{n \to \infty}{\text{plim}}$ means convergence in probability (see \cite{RR89}), and consider random fields of the form
\begin{equation}
 X(t) = \int\limits_{\R^d} f_t(x) \Lambda(dx), \quad t \in \R^q. \label{eq:int}
\end{equation}

\section{Main result}

\begin{theorem}
 The random field $X$ is infinitely divisible, that is the law of the random vector $(X(t_1),...,X(t_n))^\mathsf{T}$ is an infinitely divisible probability measure on $\R^n$ for all $t_1,...,t_n \in \R^q$.
\end{theorem}
\begin{proof}
Let $\varphi_{(t_1,...,t_n)}$ be the characteristic function of $(X(t_1),...,X(t_n))^\mathsf{T}$. It is enough to show that $\varphi_{(t_1,...,t_n)}^\gamma$ is a characteristic function for all $\gamma > 0$, cf. the proof of Theorem 3.1. in \cite{HS78}, p. 144/145. Then for each $m \in \N$, there exists a characteristic function $\varphi_m$ such that $\varphi_{(t_1,...,t_n)} = (\varphi_m)^m$. This is the corresponding condition for infinite divisibility given in \cite{CPW95}, p.~111.

Due to the linearity of the integral (\ref{eq:int}) and the fact that any linear combination of $\Lambda$-integrable functions is $\Lambda$-integrable (cf. \cite{JW94}, p. 81), we have
$$\sum_{j=1}^n x_j X(t_j) = \int\limits_{\R^d} \left(\sum\limits_{j=1}^n x_j f_{t_j}(s)\right) \Lambda(ds)$$
and the characteristic function $\varphi_{(t_1,...,t_n)}$ is given by
\begin{eqnarray*}
 \varphi_{(t_1,...,t_n)} (x) &=& \varphi_{\sum_{j=1}^n x_j X(t_j)}(1) \\
 &=& \exp \left\{ia_{\sum x_j f_{t_j}} - \frac{1}{2}b_{\sum x_j f_{t_j}} + \int\limits_{\R^d}\int\limits_{\R} c_{\sum x_j f_{t_j}}(s,y) F(ds,dy)\right\},
\end{eqnarray*}
cf. \cite{HPVJ08}, where
\begin{eqnarray*}
 a_{\sum x_j f_{t_j}} &=& \int_{\R^d} \left(\sum_{j=1}^n x_j f_{t_j}(s)\right) a(ds), \\
 b_{\sum x_j f_{t_j}} &=& \int_{\R^d} \left( \sum_{j=1}^n x_j f_{t_j}(s) \right)^2 b(ds), \\
 c_{\sum x_j f_{t_j}}(s,y) &=& e^{i\sum\limits_{j=1}^n x_j f_{t_j}(s) y}-1-i \sum_{j=1}^n x_j f_{t_j}(s) \tau(y).
\end{eqnarray*}
Let $\gamma > 0$. Then
\begin{eqnarray*}
 \varphi_{(t_1,...,t_n)}^\gamma (x) = \exp \left\{i\gamma a_{\sum x_j f_{t_j}} - \frac{1}{2}\gamma b_{\sum x_j f_{t_j}} + \int\limits_{\R^d}\int\limits_{\R} c_{\sum x_j f_{t_j}}(s,y) \gamma F(ds,dy)\right\}
\end{eqnarray*}
with
\begin{eqnarray*}
 \gamma a_{\sum x_j f_{t_j}} &=& \int_{\R^d} \left(\sum_{j=1}^n x_j f_{t_j}(s)\right) \gamma a(ds), \\
 \gamma b_{\sum x_j f_{t_j}} &=& \int_{\R^d} \left( \sum_{j=1}^n x_j f_{t_j}(s) \right)^2 \gamma b(ds).
\end{eqnarray*}

Since $a^*:=\gamma a$ is a $\sigma$-additive set function on $\mathcal{S}$, $b^*:=\gamma b$ is a measure on $\mathcal{B}(\R^d)$, and $F^*(dr,A):=\gamma F(dr,A)$ is a measure on $\mathcal{B}(\R^d)$ for fixed $dr$ and a L\'evy measure on $\mathcal{B}(\R)$ for each fixed $A \in \mathcal{B}(\R^d)$, there exists an infinitely divisible random measure $\Lambda^*$ with characteristic triplet $(a^*,b^*,F^*)$, cf. Proposition 2.1.(b) in \cite{RR89}. Therefore, $\varphi_{(t_1,...,t_n)}^\gamma$ is the characteristic function of $(Y(t_1),...,Y(t_n))^\mathsf{T}$ with
\begin{equation}
 Y(t) = \int\limits_{\R^d} f_t(x) \Lambda^*(dx). \nonumber
\end{equation}
\end{proof}

\end{document}